\documentclass[12pt,a4paper]{article}
\usepackage{etex} 
\usepackage[latin1]{inputenc}
\usepackage[T1]{fontenc}
\usepackage[right=2cm,left=2cm,top=2cm,bottom=2cm]{geometry}
\usepackage{fourier} 
\usepackage{pstricks}
\usepackage{enumerate}
\usepackage{amsmath}
\usepackage{amsthm}
\usepackage{amsfonts}
\usepackage{amssymb} 
\usepackage[np]{numprint}

\title{\normalsize \textsc{On restricted divisor sums associated to the Chowla-Walum conjecture}}
\author{\normalsize Olivier Bordell\`{e}s}
\date{}

\newcommand{\Z}{\mathbb {Z}}

\newcommand{\R}{\mathbb {R}}

\DeclareMathOperator{\mx}{max}
\DeclareMathOperator{\cw}{CW-}

\theoremstyle{theorem}
\newtheorem{prop}{Proposition}[subsection]
\newtheorem{theorem}[prop]{Theorem}
\newtheorem{coro}[prop]{Corollary}
\newtheorem{lemma}[prop]{Lemma}
\newtheorem{conj}[prop]{Conjecture}

\theoremstyle{definition}

\theoremstyle{remark}
\newtheorem{rem}[prop]{Remark}

\begin{document}

\maketitle

\thispagestyle{myheadings}
\font\rms=cmr8 
\font\its=cmti8 
\font\bfs=cmbx8

\begin{abstract}
We first study the mean value of certain restricted divisor sums involving the Chowla-Walum sums, improving in particular a recent estimate given by Iannucci. The aim of the second part of this work is the generalization of the previous study, by restricting the range of the divisors in the studied divisor sums, extending the Chowla-Walum conjecture, proving a small part of this extended conjecture and generalizing the asymptotic formulas previously obtained in the first part.
\end{abstract}

\subsection{A first restricted divisor sum}

\subsubsection{Introduction}

\noindent
Let $\alpha \geqslant 0$ be a fixed real number. For any $n \in \Z_{\geqslant 1}$, define
$$\widetilde{\sigma}_\alpha (n) = \sum_{\substack{d \mid n \\ d \leqslant \sqrt{n}}} d^\alpha$$
and set $\widetilde{\tau} := \widetilde{\sigma}_0$. In \cite[Theorem~2]{ian}, the author studies the case $\alpha =1$ and proves that
\begin{align}
   \sum_{n \leqslant x} \widetilde{\sigma}_1 (n) = \tfrac{2}{3} x^{3/2} + O (x \log x). \label{eq:iann}
\end{align}
By using a recent result of Bourgain \& Watt \cite{bouw}, we are able to generalize and improve this estimate. Our result below involves the Chowla-Walum conjecture, abbreviated here as 'the CW-conjecture`, that we recall in the next section.

\subsubsection{The Chowla-Walum conjecture}

\noindent
In what follows, let $\lfloor x \rfloor$ and $\{ x \} = x - \lfloor x \rfloor$ be the integral and fractional parts of $x \in \R$, and, for any $j \in \Z_{\geqslant 1}$, let $x \mapsto B_j \left( \{ x\} \right)$ be the $j$-th Bernoulli function, where $x \mapsto B_j(x)$ is the $j$-th Bernoulli polynomial defined inductively by setting $B_0(x) = 1$ and, for all $j \in \Z_{\geqslant 1}$
$$\frac{\textrm{d}}{\textrm{d}x} B_j(x) = j B_{j-1} (x) \quad \left( x \in \R \right) \quad \textrm{and} \quad \int_0^1 B_j(x) \, \textrm{d}x = 0.$$
It is customay to set $B_1 \left( \{x\} \right) := \psi(x) = x - \lfloor x \rfloor - \frac{1}{2}$ the first Bernoulli function. Now, for any $\alpha \geqslant 0$ and $j \in \Z_{\geqslant 1}$, define the \textit{Chowla-Walum sum}
\begin{equation}
   G_{\alpha,j}(x) := \sum_{d \leqslant \sqrt{x}} d^\alpha B_j \left( \left\lbrace \frac{x}{d} \right\rbrace \right). \label{eq:G}
\end{equation}
We extend this notation to the case $j=0$ for which we allow $\alpha$ to be any real number in this case.

\medskip

\noindent
The Dirichlet divisor conjecture states that, for $x$ large and all $\varepsilon > 0$, $G_{0,1}(x) \ll_\varepsilon x^{1/4+ \varepsilon}$. As an extension of this problem, Chowla and Walum \cite{chow} stated the following conjecture.

\begin{conj}[Chowla-Walum]
\label{conj:cw}
For $j \in \Z_{\geqslant 1}$, $\alpha \geqslant 0$, $x$ large and all $\varepsilon > 0$
$$G_{\alpha,j} (x) \ll_\varepsilon x^{\frac{1}{2} \alpha + \frac{1}{4} +\varepsilon}.$$
\end{conj}

\noindent
Chowla and Walum proved the conjecture for $(\alpha,j)=(1,2)$. Later, Kanemitsu and Sita Rama Chandra Rao \cite{kan85} proved this conjecture in the case $\alpha \geqslant \frac{1}{2}$ and $j \geqslant 2$. They also study the mean square of $G_{\alpha,j} (x)$ and proved the following bound which supports the Chowla-Walum Conjecture: if $|\alpha| \leqslant \frac{1}{2}$, then
$$\int_1^T G_{\alpha,j} (x)^2 \, \textrm{d}x \ll T^{\alpha+3/2+\varepsilon}.$$
For more results concerning the Chowla-Walum conjecture, see \cite{cao,kan78,mey85}.

\subsubsection{The main result}

\begin{theorem}
\label{th:main}
Let $\alpha \geqslant 0$ be a fixed real number and set
$$\theta_\alpha := \tfrac{1}{2} \alpha + \begin{cases} \frac{1}{4}, & \text{if the} \ \cw \text{conjecture is true} \; ; \\ & \\ \frac{517}{\np{1648}}, & \text{otherwise}. \end{cases}$$
For $x$ sufficiently large and all $\varepsilon \in \left( 0,\frac{1}{2} \right]$
$$\sum_{n \leqslant x} \widetilde{\tau} (n) = \tfrac{1}{2} x \log x + x \left( \gamma - \tfrac{1}{2} \right) + \tfrac{1}{2} x^{1/2} + O \left( x^{\theta_0 + \varepsilon} \right)$$
and, if $\alpha > 0$
$$\sum_{n \leqslant x} \widetilde{\sigma}_\alpha (n) = \tfrac{2}{\alpha(\alpha+2)} x^{1+\alpha/2} + \tfrac{1}{2(\alpha+1)} x^{(\alpha+1)/2} + x \left( \tfrac{5}{8} - \tfrac{\alpha}{8}  - \tfrac{1}{\alpha} \right)  + O \left( x^{\theta_\alpha + \varepsilon} \right).$$
\end{theorem}

\noindent
Note that in the case $\alpha > 0$, the third term is absorbed by the error term as soon as $\alpha \geqslant \frac{3}{2}$ if the $\cw$conjecture is true, and $\alpha \geqslant \frac{\np{1131}}{824} \approx \np{1.3725} \dotsc$ otherwise. When $\alpha =1$, we derive the following estimate improving \eqref{eq:iann}.

\begin{coro}
For $x$ sufficiently large and all $\varepsilon \in \left( 0,\frac{1}{2} \right]$
$$\sum_{n \leqslant x} \widetilde{\sigma}_1 (n) = \tfrac{2}{3} x^{3/2} - \tfrac{1}{4} x + O \left( x^{\theta_1 + \varepsilon} \right)$$
where
$$\theta_1 := \begin{cases} \frac{3}{4}, & \text{if the} \ \cw \text{conjecture is true} \; ; \\ & \\ \frac{\np{1341}}{\np{1648}} \approx \np{0.8137} \dotsc, & \text{otherwise}. \end{cases}$$
\end{coro}

\subsubsection{Technical tools}

\noindent
Studying the Dirichlet divisor problem, the authors \cite{bouw} prove the following important result.

\begin{lemma}
\label{le:bouw}
Let $a,b \in \Z$ such that $|a|+|b| \leqslant 1$, $N \in \Z_{\geqslant 3}$ and $x \in \R_{\geqslant 1}$ such that $3 \leqslant N \leqslant x^{1/2}$. Let $\varepsilon \in \left( 0,\frac{1}{2} \right]$. Then
$$\sum_{N < n \leqslant 2N} \psi \left( \frac{4x}{4n+a} + \frac{b}{4} \right) \ll_\varepsilon x^{\frac{517}{\np{1648}} + \varepsilon}.$$
\end{lemma}

\noindent
As a corollary, the following bound can be proved by partial summation.

\begin{coro}
\label{cor:bouw}
For $x$ sufficiently large and all $\varepsilon \in \left( 0,\frac{1}{2} \right]$
$$G_{\alpha,1}(x) \ll_\varepsilon x^{\frac{1}{2} \alpha + \frac{\np{517}}{\np{1648}} + \varepsilon}.$$
\end{coro}

\begin{proof}
By partial summation, we immediately derive using Lemma~\ref{le:bouw}
\begin{align*}
   G_{\alpha,1}(x) &\ll x^{\alpha/2} \left\lbrace \underset{K \leqslant \sqrt{x}}{\mx} \left| \sum_{d \leqslant K} \psi \left( \frac{x}{d} \right) \right|  + 1  \right\rbrace \\
   &\ll x^{\alpha/2} \left\lbrace \underset{K \leqslant \sqrt{x}}{\mx} \ \underset{3 \leqslant D \leqslant K}{\mx} \left| \sum_{D < d \leqslant 2D} \psi \left( \frac{x}{d} \right) \right| \log x  + 1 \right\rbrace  \ll_{\varepsilon} x^{\frac{1}{2} \alpha +\frac{\np{517}}{\np{1648}} + \varepsilon}
\end{align*}
as required.
\end{proof}

\noindent
The case $k=0$ in the sum \eqref{eq:G} can effectively be handled by the Euler-Maclaurin summation formula (see \cite[Theorem~B.5]{monv} for instance). We leave the details of the proof to the reader.

\begin{lemma}
\label{le:Euler_G}
Let $a \geqslant 1$. Then
$$\sum_{d \leqslant x^{1/a}} \frac{1}{d} = \tfrac{1}{a} \log x + \gamma - \psi \left( x^{1/a} \right) x^{-1/a} + O \left( x^{-2/a} \right)$$
and, if $\beta > -1$
$$\sum_{d \leqslant x^{1/a}} d^\beta = \tfrac{1}{\beta + 1} x^{(\beta+1)/a} - \psi \left( x^{1/a} \right) x^{\beta / a} + \tfrac{1}{2} - \tfrac{\beta}{8} - \tfrac{1}{\beta + 1} + O_\beta \left( x^{(\beta - 1)/a} \right).$$
\end{lemma}

\subsubsection{Proof of Theorem~\ref{th:main}}

\noindent
Interverting the summations, we derive
\begin{align*}
   \sum_{n \leqslant x} \widetilde{\sigma}_\alpha (n) &= \sum_{n \leqslant x} \ \sum_{\substack{d \mid n \\ d \leqslant \sqrt{n}}} d^\alpha \\
   &= \sum_{d \leqslant \sqrt{x}} d^\alpha \sum_{d \leqslant k \leqslant x/d} 1 \\
   &= \sum_{d \leqslant \sqrt{x}} d^\alpha \left( \left \lfloor\frac{x}{d} \right \rfloor - d + 1 \right)  = \sum_{d \leqslant \sqrt{x}} d^\alpha \left( \frac{x}{d} - d + \frac{1}{2}  - \psi \left( \frac{x}{d} \right) \right) \\
   &= x G_{\alpha-1,0}(x) - G_{\alpha+1,0}(x) + \tfrac{1}{2} G_{\alpha,0}(x) - G_{\alpha,1}(x)
\end{align*}
where $G_{\alpha,j}(x)$ is defined in \eqref{eq:G}. The proof follows from Lemma~\ref{le:Euler_G} with $a=2$, Corollary~\ref{cor:bouw} and Conjecture~\ref{conj:cw}. Note that the identity $\widetilde{\tau} (n) = \frac{1}{2} \left( \tau(n) + \mathbf{1}_\square (n) \right)$, where $\mathbf{1}_\square (n) = 1$ if $n$ is a perfect square and $0$ otherwise, yields an alternative proof of the first estimate of Theorem~\ref{th:main}.
\qed

\subsection{A second restricted divisor sum}

\subsubsection{Introduction}

\noindent
In this section, let $a > 1$ and $\alpha \geqslant 0$ be fixed real numbers. For any $n \in \Z_{\geqslant 1}$, define
$$\sigma_{a,\alpha} (n) = \sum_{\substack{d \mid n \\ d \leqslant n^{1/a}}} d^\alpha$$
and set $\tau_{a} (n) := \sigma_{a,0} (n)$. As above, the study of the mean value
$$\sum_{n \leqslant x} \sigma_{a,\alpha} (n)$$
involve sums, defined in \eqref{eq:G2} below, generalizing the Chowla-Walum sums given in \eqref{eq:G}.

\subsubsection{A slight extension of the $\cw$conjecture}

\noindent
Let $a > 1$, $\alpha \geqslant 0$ and $j \in \Z_{\geqslant 1}$ and set
\begin{equation}
   G_{a,\alpha,j}(x) := \sum_{d \leqslant x^{1/a}} d^\alpha B_j \left( \left\lbrace \frac{x}{d} \right\rbrace \right). \label{eq:G2}
\end{equation}
As in \eqref{eq:G}, we allow $\alpha$ to be any real number without restriction in the case $j=0$. 

\medskip

\noindent
Extending Conjecture~\ref{conj:cw}, we surmise the following estimate.

\begin{conj}
\label{conj:cw2}
For $j \in \Z_{\geqslant 1}$,  $a > 1$, $\alpha \geqslant 0$, $\varepsilon > 0$ and $x$ large
$$G_{a,\alpha,j} (x) \ll_\varepsilon x^{\frac{1}{a} \alpha + \frac{1}{2a} +\varepsilon}.$$
\end{conj}

\noindent
We have not been able to prove this bound, but, as a step towards this conjecture, we will show the following estimates.

\begin{theorem}
\label{th:CW2}
Let $a > 1$, $\alpha \geqslant 0$ and $j \in \Z_{\geqslant 1}$. For all $\varepsilon > 0$ and $x$ sufficiently large
$$G_{a,\alpha,1} (x) \ll_\varepsilon x^{\frac{1}{a} \alpha + \frac{34}{283a} + \frac{76}{283} +\varepsilon} + x^{\frac{1}{a} \alpha + \frac{2}{a}-1} \log x$$
and, if $j \geqslant 2$
$$G_{a,\alpha,j} (x) \ll_\varepsilon x^{\frac{1}{a} \alpha + \frac{55}{194} +\varepsilon} + x^{\frac{1}{a} \alpha + \frac{2}{a}-1} \log x.$$
\end{theorem}

\begin{rem}
This result shows that, in the case $j \geqslant 2$, Conjecture~\ref{conj:cw2} is settled provided that $\frac{3}{2} \leqslant a \leqslant \frac{97}{55}$.
\end{rem}

\begin{proof}
For $N \in \Z_{\geqslant 1}$ large, $\alpha \geqslant 0$ and $j \in \Z_{\geqslant 1}$, define
$$\mathcal{G}_{N,\alpha,j} := \sum_{N < d \leqslant 2N} d^\alpha B_j \left( \left\lbrace \frac{x}{d} \right\rbrace \right).$$
The usual splitting argument yields
$$G_{a,\alpha,j}(x) \ll \underset{N \leqslant x^{1/a}}{\mx} \left( \mathcal{G}_{N,\alpha,j} \right) \log x$$ 
so that it suffices to bound $\mathcal{G}_{N,\alpha,j}$. Assume first $j=1$. Then by partial summation and \cite[Corollary~6.35]{bor}, we get if $(k, \ell)$ is an exponent pair
$$\mathcal{G}_{N,\alpha,1} \ll N^\alpha \underset{N \leqslant N_1 \leqslant 2N}{\mx} \left| \sum_{N < d \leqslant N_1} \psi \left( \frac{x}{d} \right) \right| \ll x^{\frac{k}{k+1}} N^{\alpha + \frac{\ell-k}{k+1}} + N^{2+\alpha} x^{-1}.$$
If $j \geqslant 2$, recall that, for any $t \in \R$
$$B_j \left( \{ t \} \right) = - \frac{j!}{(2 \pi i)^j} \sum_{m \neq 0} \frac{e(mt)}{m^j}$$
where, as usual, $e(x) := e^{2 \pi i x}$, so that
\begin{align*}
   \mathcal{G}_{N,\alpha,j} & \ll N^\alpha \sum_{m \geqslant 1} m^{-j} \underset{N \leqslant N_1 \leqslant 2N}{\mx} \left| \sum_{N < d \leqslant N_1} e\left( \frac{mx}{d} \right)  \right| \\
   & \ll N^\alpha \sum_{m \geqslant 1} m^{-j} \left\lbrace \left( \frac{mx}{N} \right)^k N^{\ell - k} + N^2 (mx)^{-1} \right\rbrace \\
   & \ll x^k N^{\alpha + \ell - 2 k} \sum_{m \geqslant 1} m^{k-j} + N^{2+\alpha} x^{-1} \sum_{m \geqslant 1} m^{-j-1} \\
   & \ll x^k N^{\alpha + \ell - 2 k} + N^{2+\alpha} x^{-1}
\end{align*}
since $k - j \leqslant \frac{1}{2} - 2 = - \frac{3}{2}$. Hence, since $\alpha(k+1) + \ell - k \geqslant 0$
$$G_{a,\alpha,1} (x) \ll \left( x^{\frac{1}{a} \alpha + \frac{k(a-1)+\ell}{a(k+1)}} + x^{\frac{\alpha+2}{a}-1} \right) \log x$$
and, if $j \geqslant 2$ and $\alpha + \ell - 2 k \geqslant 0$
$$G_{a,\alpha,j} (x) \ll \left( x^{\frac{1}{a} \alpha + \frac{k(a-2)+\ell}{a}} + x^{\frac{\alpha+2}{a}-1} \right) \log x$$
and the result follows by using Bourgain's exponent pair \cite[Theorem~6]{bou} $(k, \ell) = BA^2\left( \frac{13}{84} + \varepsilon, \frac{55}{84} + \varepsilon \right) = \left( \frac{76}{207} + \varepsilon, \frac{110}{207} + \varepsilon \right)$ if $j=1$, and $(k, \ell) = BA \left( \frac{13}{84} + \varepsilon, \frac{55}{84} + \varepsilon \right) = \left( \frac{55}{194} + \varepsilon, \frac{55}{97} + \varepsilon \right)$ if $j \geqslant 2$. 
\end{proof}

\subsubsection{Application}

\begin{theorem}
\label{th:main2}
Let $\alpha \geqslant 0$ be a fixed real number and $a \in \Z_{\geqslant 3}$. For $x$ sufficiently large
$$\sum_{n \leqslant x} \tau_{a} (n) = \tfrac{1}{a} x \log x + x \left( \gamma - \tfrac{1}{a} \right)  + O \left( x^{1-2/a} \right)$$
and, if $\alpha > 0$
$$\sum_{n \leqslant x} \sigma_{a,\alpha} (n) = \tfrac{a}{\alpha(\alpha+a)} x^{1+\alpha/a} + x \left( \tfrac{5}{8} - \tfrac{\alpha}{8}  - \tfrac{1}{\alpha} \right)  + O \left( x^{1+(\alpha-2)/a} \right).$$
\end{theorem}

\begin{proof}
As in the proof of Theorem~\ref{th:main}, we derive
$$\sum_{n \leqslant x} \sigma_{a,\alpha} (n) = x G_{a,\alpha-1,0} (x) - G_{a,\alpha+a-1,0}(x) + \tfrac{1}{2} G_{a,\alpha,0}(x) - G_{a,\alpha,1}(x)$$
since $a \geqslant 3$ is now an integer, and the use of Lemma~\ref{le:Euler_G} and Theorem~\ref{th:CW2} yields the asserted estimates with an extra error term given by
$$x^{\frac{\alpha}{a} + \frac{34}{283a} + \frac{76}{283} + \varepsilon}$$
which is easily seen to be absorbed by the other error term since $a \geqslant 3$. 
\end{proof}

\noindent
With $\alpha = 1$ and $a \in \Z_{\geqslant 3}$, we derive the following formula.

\begin{coro}
Let $a \in \Z_{\geqslant 3}$. For $x$ sufficiently large
$$\sum_{n \leqslant x} \sigma_{a,1} (n) = \tfrac{a}{a+1} x^{1+1/a} - \tfrac{1}{2} x + O \left( x^{1-1/a} \right).$$
\end{coro}

\end{document}